\documentclass[11pt]{elsarticle}

\usepackage[utf8]{inputenc}
\usepackage[english]{babel}
\usepackage{enumerate}
\usepackage{latexsym}
\usepackage{amsthm,amsmath}
\usepackage{amssymb}
\usepackage{aliascnt} 
\usepackage{multicol}
\usepackage{tikz}
\usepackage{float} 

\usepackage[bookmarks]{hyperref} 
\hypersetup{
	pdftitle={2-switch-degree Classification of Split Graphs},
	pdfauthor={V.N. Schvollner},
	colorlinks=true,
	linkcolor=blue,
	citecolor=red,
	filecolor=cyan,
	urlcolor=magenta
}

\usetikzlibrary{babel,decorations.pathreplacing,calc,positioning,shapes.geometric}

\usepackage{cleveref}

\definecolor{10}{RGB}{115,59,171}
\definecolor{8}{RGB}{212,122,240}
\definecolor{7}{RGB}{99,212,119}
\definecolor{6}{RGB}{183,240,164}
\definecolor{D}{RGB}{255,162,79}
\definecolor{E}{RGB}{255,84,0}
\definecolor{F}{RGB}{158,248,255}
\definecolor{G}{RGB}{128,135,255}
\definecolor{I}{RGB}{187,255,0}
\definecolor{A}{cmyk}{.9,.05,.4,0}
\definecolor{B}{RGB}{150,30,150}
\definecolor{C}{RGB}{186,155,189}
\definecolor{9}{RGB}{0,180,60}
\definecolor{0}{RGB}{30,123,191}
\definecolor{1}{RGB}{255,113,102}
\definecolor{2}{RGB}{41,199,92}
\definecolor{3}{RGB}{242,207,16}
\definecolor{5}{RGB}{255,15,154}
\definecolor{4}{rgb}{.8,0,.8}

\definecolor{Red}{rgb}{1,0.4,0.4}
\definecolor{Green}{rgb}{.1,.5,.1}
\definecolor{Blue}{rgb}{.1,.1,.5}
\definecolor{blue}{RGB}{0,0,255}
\definecolor{Yellow}{rgb}{.8,.4,0}
\definecolor{X}{rgb}{.8,.4,0}
\definecolor{H}{rgb}{0,0,1}
\definecolor{light}{rgb}{.67,.84,.90}
\definecolor{Cyan}{rgb}{0,1,1}
\definecolor{Purple}{rgb}{.5,0,.5}
\definecolor{Purple2}{rgb}{.5,.2,.5}
\definecolor{white}{rgb}{1.0,1.0,1.0}
\definecolor{Purple2}{rgb}{.8,.4,0}
\definecolor{Amarillo}{RGB}{225,191,73}
\definecolor{Celeste}{RGB}{117,170,219}
\definecolor{Castano}{RGB}{232,53,17}
\definecolor{Black}{RGB}{0,0,0}
\definecolor{White}{RGB}{255,255,255}
\definecolor{gris}{rgb}{.5,.5,.5}

\newtheorem{theorem}{Theorem}[section]
\newaliascnt{corollary}{theorem}
\newtheorem{corollary}[corollary]{Corollary}
\aliascntresetthe{corollary}
\newaliascnt{lemma}{theorem}
\newtheorem{lemma}[lemma]{Lemma}
\aliascntresetthe{lemma}
\newaliascnt{proposition}{theorem}
\newtheorem{proposition}[proposition]{Proposition}
\aliascntresetthe{proposition}

\pgfdeclarelayer{background2}
\pgfdeclarelayer{background}
\pgfdeclarelayer{foreground}
\pgfsetlayers{background2,background,main,foreground}

\begin{document}
	\begin{frontmatter}
		\title {
			Induced paths and cycles in factor graphs \\ of split graphs
		}
		\address[IMASL]{Instituto de Matem\'atica Aplicada San Luis - CONICET, UNSL (San Luis, Argentina)}
		\address[DEPTO]{Departamento de Matem\'atica, Universidad Nacional de San Luis, San Luis, Argentina.}

		\author[IMASL,DEPTO]{Victor N. Schvöllner}\ead{vnsi9m6@gmail.com} 
		\author[IMASL]{Adri\'an Pastine}\ead{agpastine@unsl.edu.ar}

		\begin{abstract}
			Let $S$ be a split graph with bipartition $(K,I)$ and let $\Phi(S)$ be the factor graph associated with $S$, a multigraph on $I$ whose encodes the combinatorial information about 2-switch transformations in $S$.	We study induced paths and cycles in $\Phi(S)$ and show that they impose strong structural restrictions on the neighborhoods in $S$ of the corresponding vertices. In particular, induced paths generate chains of neighborhood inclusions which force a monotone behavior of the degrees (in $S$) of their vertices along the path. As a consequence, we prove that induced cycles in $\Phi(S)$ have length $\leq 4$. Finally, we show that in any induced path only the first or the last edge can be simple, which yields an upper bound for the diameter of $\Phi(S)$ in terms of the 2-switch-degree of $S$.
%
		\end{abstract}
		
	\end{frontmatter}	

\section{Introduction}\label{sec_intro}

Let $G$ be a graph. We use $V(G)$ and $E(G)$ to refer to the vertex set and the edge set of $G$, respectively. As usual, $|G|$ denotes the order of $G$.
The degree sequence of a graph $G$ with $V(G)=[n]=\{1,\dots,n\}$ is the $n$-tuple $(d_1,\dots,d_n)$, where $d_v$ is the degree of vertex $v$ in $G$. \\

Structural transformations on graphs play a fundamental role in the study of graph classes defined by degree constraints. Among these operations, the \emph{2-switch} is one of the most important tools for navigating the family of graphs that share a common degree sequence. Given vertices $a,b,c,d$ such that $ab,cd\in E(G)$ and $ac,bd\notin E(G)$, the 2-switch replaces the edges $ab,cd$ with $ac,bd$. This operation preserves the degree sequence of the graph and has been widely used in several contexts. For closely related results, see \cite{arikati1999realization,pastine20252,schvollner20262,taylor2006contrained}. 

A classical result highlighting the importance of these transformations is the canonical decomposition theorem of Tyshkevich (see \cite{tyshkevich2000decomposition}), which describes how graphs can be decomposed into indecomposable factors with respect to certain composition operations. In particular, indecomposable graphs appear as fundamental building blocks from which large classes of graphs can be constructed. Understanding the structure of these components and the transformations acting on them is therefore a natural and important problem.

In this direction, Barrus and West introduced in \cite{barrus.west.A4} the graph $A_4(G)$ associated with a graph $G$. The vertices of $A_4(G)$ are the vertices of $G$, and two vertices are adjacent whenever they participate in the same 2-switch of $G$. One of their main results shows that $A_4(G)$ decomposes as the disjoint union of the corresponding $A_4$-graphs of the Tyshkevich factors of $G$. As a consequence, a graph $G$ is indecomposable if and only if $A_4(G)$ is connected. This connection between 2-switch transformations and structural decomposition provides a useful framework for studying graph realizations.

Split graphs constitute a particularly well-behaved class within this context. A graph $G$ is called a \emph{split graph} if its vertex set can be partitioned as $K\dot\cup I$, where $K$ is a clique and $I$
is an independent set. When such a partition is fixed, we write $(S,K,I)$ to emphasize that the split graph $S$ is considered together with a specific bipartition. Split graphs appear naturally in
several areas of graph theory and are widely studied in the literature (see \cite{splitnordhausgaddum,hammer1977split,jaume2025nullspace,whitman2020split}).

Motivated by the study of 2-switch dynamics in split graphs and inspired by Barrus and West's $A_4$, in \cite{pastine2025simple} the authors introduced the \emph{factor graph} of a split graph $(S,K,I)$, denoted by $\Phi(S)$, or simply by $\Phi$, if $S$ is clear form the context. This is a loopless multigraph with $V(\Phi)=I$, such that 
there is an edge joining $u$ and $v$ in $\Phi$ for each 2-switch acting on $u$ and $v$. Thus, $E(\Phi)$ is a multiset and each edge of $\Phi$ has a multiplicity. 
In this way, $\Phi(S)$ encodes the combinatorial information associated with 2-switch transformations on $S$. In particular, the size of $\Phi(S)$ (counting multiplicities) equals the \textit{2-switch-degree} of $S$, i.e., the number of 2-switches acting on $S$ (see \cite{pastine20252}). Moreover, $\Phi(S)$ avoids certain redundancies present in the construction of $A_4(S)$, providing a more compact description of the 2-switch structure in the special case of split graphs.   

Several structural properties of $\Phi$ were established in \cite{pastine2025simple}, including relationships between the multiplicities and the neighborhoods in $S$ of the corresponding vertices. In particular, these results show that the factor graph captures delicate inclusion relations between neighborhoods in the split graph and can be used to obtain structural information about $S$, especially when $\Phi(S)$ is simple or connected.

The purpose of this article is to further investigate the structure of $\Phi$ by focusing on its induced paths and cycles. Understanding these configurations reveals strong constraints on how
neighborhoods in the split graph may be arranged. 

This article is divided into 4 sections. In \Cref{sec:ind_paths_Phi} we study induced paths in $\Phi$. We note that the presence of an induced path $P$ in $\Phi$ determines certain inclusion chains in the neighborhoods in $S$ of the vertices of $P$. Thanks to this, we also see how the degree in $S$ of the vertices of $P$ tends to increase along $P$, moving from one end to the other. These results are then applied in \Cref{sec:ind_cycles_Phi} to study the length of induced cycles in $\Phi$ (since these contain induced paths), establishing that it is at most 4. Subsequently, in \Cref{sec:upp_bound_diam_Phi}, we return to induced paths, now analyzing the multiplicity of their edges. We find that only the first or the last edge can be simple (i.e., of multiplicity 1). This result allows us to conclude the section with an interesting upper bound for the diameter of $\Phi(S)$, in terms of the 2-switch-degree of $S$. Finally, in \Cref{sec_conclusion}, we summarize the main results of the paper and discuss possible directions for future research.


\section{Induced paths in $\Phi(S)$}
\label{sec:ind_paths_Phi}

From now on, we only concern about a split graph $S$ and its associated factor graph $\Phi(S)$. So, the symbol $N_i$ will always refer to the (open) neighborhood in $S$ of the vertex $i$ or $v_i$. Similarly, the symbol $d_i$ will always refer to the degree in $S$ of the vertex $i$ or $v_i$. On the other hand, we denote by $N_{\Phi}(i)$ the (open) neighborhood in $\Phi(S)$ of the vertex $i$ or $v_i$, and by $\sigma_{uv}$ the multiplicity of the edge $uv$ in $\Phi(S)$. 

In \cite{pastine2025simple} the authors show that 
\[ \sigma_{uv}=(d_u-\eta_{uv})(d_v-\eta_{uv}), \]
where $\eta_{uv}=|N_u\cap N_v|$, and subsequently derive from it the following proposition. 

\begin{proposition}[\cite{pastine2025simple}]
	\label{prop.basicas.sigma_uv}
	Let $S$ be a split graph. Then:
	\begin{enumerate}
		\item $\sigma_{uv}=0$ (i.e., $uv\notin E(\Phi)$) and $d_v \leq d_u$ if and only if $N_v \subset N_u$; 
		
		\item $\sigma_{uv}=0$ and $d_u =d_v$ if and only if $N_u = N_v$;
		
		\item if $N_u = N_v $, then $N_{\Phi}(u)=N_{\Phi}(v)$;
		
		\item if $\sigma_{uv}=1$, then $d_u =d_v$ and $|N_u -N_v |=|N_v -N_u |=1$.
		
		
		
		
	\end{enumerate}
\end{proposition}

To streamline the presentation, we omit explicit references to items (1) and (2) of \Cref{prop.basicas.sigma_uv}, as they are ubiquitous in all the proofs of this article. So, we will limit our citations to items (3) and (4).

\begin{lemma}
	\label{lema.camino.inducido.en.Phi}
	Let $(S,K,I)$ be a split graph and $P=v_1\ldots v_n$ an induced path in $\Phi(S)$, with $n\geq 2$. If $d_1=\max\{d_i:i\in[n]\}$, then for each $i\in[n]$ we have
	\begin{equation*}
		d_i\geq d_j \ \text{and} \ N_i\supseteq N_j \ \text{for all} \ j\geq i+2.
	\end{equation*}
\end{lemma}

\begin{proof}
	We proceed by induction on $i$. Since $\sigma_{1j}=0$ and $d_1 \geq d_j$ for all $j\geq 3$, we have $N_1\supseteq N_j$, for all $j\geq 3$. For $i\geq 2$, we assume that $d_{i-1}\geq d_j$ and $N_{i-1}\supseteq N_j$, for all $j\geq i+1$. If for some $k\geq i+2$ we had $d_i<d_k$, then we would have $N_i\subset N_k\subset N_{i-1}$, where the last containment follows from the inductive hypothesis. This would imply that $\sigma_{i-1,i}=0$, which contradicts that $P$ is a path.
\end{proof}

\begin{lemma}
	\label{no.existe.Phi=P_5.tal.que...}
	There is no split graph $S$ such that $\Phi(S)$ contains an induced path $v_5 v_4 v_1 v_2 v_3$ where $\max\{ d_i:i\in[5]\}=d_1$.
\end{lemma}

\begin{proof}
	Suppose such a split graph $S$ exists. Then: $N_5,N_3\subset N_1$, since $\sigma_{13}=0=\sigma_{15}$ and $d_3,d_5\leq d_1$. As $\sigma_{24}=0$, we can assume without loss of generality that $N_4\subset N_2$. Then, it must be that $N_4\subset N_3$, since otherwise $N_3\subset N_4$ would imply $\sigma_{23}=0$. But then $N_4\subset N_1$, and therefore $\sigma_{14}=0$, which is a contradiction.
\end{proof}

\begin{theorem}
	\label{caminos.inducidos.en.Phi}
	Let $S$ be a split graph and $P=v_1\ldots v_n$ an induced path in $\Phi(S)$, with $n\geq 2$. Then,
	\begin{equation}
		\max\{ d_i:i\in[n] \}\in\{d_1,d_2,d_{n-1},d_n\}.
		\label{max.d_i}
	\end{equation}
	Moreover, if $\max\{ d_i:i\in[n] \}=d_1$, then:
	\begin{enumerate}
		\item for each $i\geq 1$: $d_i\geq d_j \ \text{and} \ N_i\supseteq N_j \ \text{for all} \ j\geq i+2$;
		\item for each $i\geq 1$: $N_i\supseteq\bigcup_{j= i+2}^n N_j$; in particular, $\bigcup_{i=1}^n N_i=N_1\cup N_2$;
		\item $d_{2k}\geq d_{2k+2}$ and $d_{2k-1}\geq d_{2k+1}$ for all $k\geq 1$;
		\item $\min\{ d_i:i\in[n] \}\in\{d_{n-1},d_n\}$.
	\end{enumerate} 
\end{theorem}

The notation $H\prec G$, which first appears in the following proof, means that $H$ is an induced subgraph of $G$.

\begin{proof}
	We begin by proving \eqref{max.d_i}. If $n\leq 4$, there is nothing to do. If $n\geq 5$, suppose that $\max\{ d_i:i\in[n] \}=d_j$, where $j\notin\{1,2,n-1,n\}$. Then $v_{j-2}v_{j-1}v_jv_{j+1}v_{j+2}\prec \Phi(S)$, but by \Cref{no.existe.Phi=P_5.tal.que...} we know this is impossible. 
	\begin{enumerate}[(1).]
		\item It is the content of \Cref{lema.camino.inducido.en.Phi}.
		\item It is an immediate consequence of (1).
		\item From (1) it immediately follows that $d_1\geq d_3\geq d_5\geq d_7\geq\ldots$ and that $d_2\geq d_4\geq d_6\geq d_8\geq\ldots$.
		\item If $\min\{ d_i:i\in[n] \}=d_j$ for some $j<n-1$, then $d_j\leq d_n$, which contradicts (1).
	\end{enumerate}  
\end{proof}

The following is a direct application of \Cref{caminos.inducidos.en.Phi}. Through it, we find an interesting divisibility relationship involving the multiplicity of one of the terminal edges of an induced path in $\Phi$. 

\begin{corollary}
	\label{K-d_max.divide.sigma}
	Let $(S,K,I)$ be a split graph such that $K=\bigcup_{v\in I}N_v$. If 
	$P=v_1\ldots v_n\prec\Phi(S)$ and $d_1=\max\{d_i:i\in[n],n\geq 2\}$, then we have the following.
	\begin{enumerate}
		\item $|\bigcup_{i=1}^n N_i|-d_1$ divides $\sigma_{12}$.
		
		\item If $\Phi=P$, then $|K|-d_1$ divides $\sigma_{12}$.
	\end{enumerate}
\end{corollary}

\begin{proof}
	Since $P$ is an induced path in $\Phi$, we have that $\bigcup_{i=1}^n N_i$ $=N_1\cup N_2$, by \Cref{caminos.inducidos.en.Phi}. Then, $|\bigcup_{i=1}^n N_i|-d_1=d_2-\eta_{12}$. We complete the proof using that $\sigma_{12}=(d_1-\eta_{12})(d_2-\eta_{12})$.
\end{proof}

Notice that from \Cref{K-d_max.divide.sigma} we immediately deduce that 
\[ \left|\bigcup_{i=1}^n N_i \right|\leq d_1+\sqrt{\sigma_{12}}. \] 


\section{Induced cycles in $\Phi(S)$}
\label{sec:ind_cycles_Phi}

Since every induced cycle in a graph (or multigraph) contains induced paths, we can apply \Cref{caminos.inducidos.en.Phi} to such paths in $\Phi$. In this way, we obtain an interesting result: $\Phi$ can only contain induced cycles of length 3 or 4. This is precisely the content of \Cref{ciclos.inducidos.en.Phi}, whose proof is essentially divided between the following two lemmas. 

\begin{lemma}
	\label{no.hay.C_5.en.Phi.con.|C|<6}
	If $S$ is a split graph and $C$ is an induced cycle in $\Phi(S)$, then $|C|\leq 5$.
\end{lemma}

\begin{proof}
	Assume that $|C|\geq 6$. If $C=v_1\ldots v_nv_1$, we can assume without loss of generality that $d_1=\max\{d_i:i\in[n]\}$. Then, $P=v_1\ldots v_{n-1}$ and $P'=v_1v_n v_{n-1}\ldots v_3$ are induced paths of $\Phi$. Since $n\geq 6$, we obtain that $N_3\supseteq N_{n-1}$ and $N_{n-1}\supseteq N_3$, applying \Cref{caminos.inducidos.en.Phi} to $P$ and $P'$ respectively. Consequently, $N_3=N_{n-1}$, and therefore $N_{\Phi}(3)=N_{\Phi}(n-1)$, by \Cref{prop.basicas.sigma_uv}(3). This means that $v_2v_{n-1} \in E(\Phi)$, which is a contradiction.
\end{proof}

\begin{lemma}
	\label{no.hay.C_5.en.Phi}
	If $S$ is a split graph and $C$ is an induced cycle in $\Phi(S)$, then $|C|\neq 5$.
\end{lemma}

\begin{proof}
	Suppose that $\Phi(S)$ has an induced cycle $C$ of length 5. If $C=v_1\ldots v_5v_1$, we can assume without loss of generality that $\max\{d_i\in[5]\}=d_1$. Since $v_1v_2v_3v_4$ and $v_1v_5v_4v_3$ are induced paths in $S$, we can use \Cref{caminos.inducidos.en.Phi} to infer that $N_4\subset N_2$ and that $N_3\subset N_5$. Since $\sigma_{25}=0$, we have $N_2\subset N_5$ or $N_5\subset N_2$. If $N_2\subset N_5$, then $N_4\subset N_2$, and hence $\sigma_{45}=0$. If $N_5\subset N_2$, then $N_3\subset N_5$, which implies $\sigma_{23}=0$. Both are contradictions.
\end{proof}

\begin{theorem}
	\label{ciclos.inducidos.en.Phi}
	If $S$ is a split graph and $C$ is an induced cycle in $\Phi(S)$, then $|C|\leq 4$.
\end{theorem}

\begin{proof}
	It follows directly from \Cref{no.hay.C_5.en.Phi.con.|C|<6,no.hay.C_5.en.Phi}.
\end{proof}


\section{An upper bound for the diameter of $\Phi(S)$}
\label{sec:upp_bound_diam_Phi}

As announced in \Cref{sec_intro}, we will now see an interesting result about the simple edges of a $P_4\prec\Phi$. 

\begin{theorem}
	\label{prohibido.P_4.sigma_23=1}
	If $S$ is a split graph, then $\Phi(S)$ cannot contain an induced path $v_1 v_2 v_3 v_4$ where $\sigma_{23}=1$.
\end{theorem}

The proof of \Cref{prohibido.P_4.sigma_23=1} requires several steps. Therefore, we postpone it momentarily, as it will be obtained more easily by first proving some preliminary lemmas.  

\begin{lemma}
	\label{prohibido.P_4.sigma_23=1.d_1<d_2>d_4}
	If $S$ is a split graph, then $\Phi(S)$ cannot contain an induced path $v_1 v_2 v_3 v_4$ where $\sigma_{23}=1$ and $d_1\leq d_2\geq d_4$.
\end{lemma}

\begin{proof}
	Suppose there exists a split graph $S$ such that $v_1v_2v_3v_4\prec\Phi=\Phi(S), \sigma_{23}=1$ and $d_1\leq d_2\geq d_4$. Since $\sigma_{23}=1$, by \Cref{prop.basicas.sigma_uv}(4) we have that $d_2=d_3$, which implies that $d_3\geq d_4$ and $N_1\subset N_3$. Since $\sigma_{14}=0$, we can assume without loss of generality that $d_4\leq d_1$ and that $N_4\subset N_1\subset N_3$. But this contradicts that $\sigma_{34}\neq 0$. 
\end{proof}

\begin{lemma}
	\label{prohibido.P_4.sigma_23=1.d_1>d_2<d_4}
	If $S$ is a split graph, then $\Phi(S)$ cannot contain an induced path $v_1 v_2 v_3 v_4$ where $\sigma_{23}=1$ and $d_1\geq d_2\leq d_4$.
\end{lemma}

\begin{proof}
	The proof is entirely analogous to that of \Cref{prohibido.P_4.sigma_23=1.d_1<d_2>d_4}.
\end{proof}

The following proposition allows, among other things, to obtain a quick proof of \Cref{prohibido.P_4.sigma_23=1.d_1<d_2<d_4}. 

\begin{proposition}
	\label{P_3.d_1<d_2.sigma_23=1.implica...}
	Let $S$ be a split graph. If $\Phi(S)$ contains an induced path $v_1v_2v_3$ such that $d_1\leq d_2$ and $\sigma_{23}=1$, then:
	\begin{enumerate}
		\item $N_1-N_2=\{x\}=N_3-N_2$, for some vertex $x$ in $S$.
		\item $N_3=\{x\}\dot{\cup}(N_2\cap N_3)\subsetneq N_1\cup N_2$.
		\item $\sigma_{12}=d_2-d_1+1$.
	\end{enumerate} 
\end{proposition}

\begin{proof}
	\begin{enumerate}[(1).]
		\item Since $\sigma_{23}=1$, we have $d_2=d_3$ and $|N_2-N_3|=1=|N_3-N_2|$, by \Cref{prop.basicas.sigma_uv}(4). Moreover, $N_1\subset N_3$, since $\sigma_{13}=0$ and $d_1\leq d_3$. If $|N_1-N_2|>1$, it is clear that we would also have $|N_3-N_2|>1$, since $N_1\subset N_3$. Therefore, it must be $N_1-N_2=\{x\}=N_3-N_2$, for some vertex $x$ in $S$.
		
		\item Since $N_3-N_2=N_3-N_2\cap N_3$, it follows that $N_3=\{x\}\dot{\cup}(N_2\cap N_3)$. In particular, $N_3\subset N_1\cup N_2$, and the inclusion is strict because $N_2-N_3\neq\varnothing$.
		
		\item Since $\sigma_{12}=(d_1-\eta_{12})(d_2-\eta_{12})$ and $d_1-\eta_{12}=|N_1-N_2|=1$, we easily obtain that $\sigma_{12}=d_2-d_1+1$.
	\end{enumerate}
\end{proof}

\begin{lemma}
	\label{prohibido.P_4.sigma_23=1.d_1<d_2<d_4}
	If $S$ is a split graph, then $\Phi(S)$ cannot contain an induced path $v_1 v_2 v_3 v_4$ where $\sigma_{23}=1$ and $d_1\leq d_2\leq d_4$.
\end{lemma}

\begin{proof}
	Suppose such a path exists in $\Phi(S)$. Since $d_1\leq d_2=d_3\leq d_4$ and $\sigma_{14}=\sigma_{24}=0$, it follows that $N_1\cup N_2\subset N_4$. But, thanks to \Cref{P_3.d_1<d_2.sigma_23=1.implica...}, we know that $N_3\subset N_1\cup N_2$. Thus, $\sigma_{34}=0$, a contradiction.
\end{proof}

We are now ready to prove \Cref{prohibido.P_4.sigma_23=1}.

\begin{proof}[Proof of \Cref{prohibido.P_4.sigma_23=1}]
	Suppose there exists a split graph $S$ such that
	\[ v_1v_2v_3v_4\prec\Phi(S) \ \text{and} \ \sigma_{23}=1. \]
	Since $d_2=d_3$ by \Cref{prop.basicas.sigma_uv}(4), we only have 3 possibilities for the degrees of $v_1, v_2$ and $v_4$ in $S$: 1) $d_1\leq d_2\geq d_4$; 2) $d_1\geq d_2\leq d_4$; 3) $d_1\leq d_2\leq d_4$. However, these cases conflict with \Cref{prohibido.P_4.sigma_23=1.d_1<d_2>d_4,prohibido.P_4.sigma_23=1.d_1>d_2<d_4,prohibido.P_4.sigma_23=1.d_1<d_2<d_4}.
\end{proof}

An important immediate consequence of \Cref{prohibido.P_4.sigma_23=1} is that, in any induced path of $\Phi$, only the first or last edge can be simple. This possibility is indeed realized; that is: there exist split graphs $S$ such that $\Phi(S)$ contains an induced $v_1v_2v_3v_4$ with $\sigma_{12}=1$. As an example, we can take the split graph $(S,\{x,y,z,t\},[4])$ where $N_1=\{x\}, N_2=\{y\}, N_3=\{x,z\}$ and $N_4=\{x,y,t\}$ (in this case: $\Phi\approx P_4, \sigma_{12}=1$ and $\sigma_{23}=2=\sigma_{34}$).

\begin{corollary}
	\label{no.aristas.simples.internas.en.P_n}
	Let $S$ be a split graph and let $v_1\ldots v_n$ ($n\geq 2$) be an induced path in $\Phi(S)$. If $\sigma_{i,i+1}=1$ for some $i\in[n-1]$, then $i\in\{1,n-1\}$. 
\end{corollary}

\begin{proof}
	It follows immediately from \Cref{prohibido.P_4.sigma_23=1}.
\end{proof}

In this article, distances between vertices in a multigraph are measured by ignoring edge multiplicities. In other words, the metric in $\Phi$ is the usual path-metric of a simple graph. On one hand, the fact that paths maximize the diameter of connected graphs with a fixed size is a classical result. On the other hand, recall from \Cref{sec_intro} that, counting edge-multiplicities, the size of $\Phi(S)$ equals $\deg(S)$, i.e., the 2-switch-degree of $S$. Hence, the diameter of a connected $\Phi(S)$ cannot exceed the length of a path multigraph $P$ with size $\deg(S)$. Clearly, the multiplicity of each edge of $P$ should be as small as possible, in order to maximize its length. Therefore, using \Cref{no.aristas.simples.internas.en.P_n}, it is easy to conclude that the internal edges of such a $P$ with maximum length, must all have multiplicity 2, while its terminal edges (i.e., the first and last) should have multiplicity 1 or 2. We have thus obtained the following inequality.

\begin{theorem}
	\label{diam_deg_bound}
	Let $S$ be a split graph. If $\Phi(S)$ is connected, then its diameter is at most $\lceil (\deg(S)+1)/2 \rceil$.
\end{theorem}

\begin{proof}
	It follows from the previous discussion.
\end{proof}		

The inequality obtained in \Cref{diam_deg_bound} is sharp. In fact, it is not difficult to see that, for every positive integer $n$, we can always construct a split graph $S_n$ such that:
\begin{enumerate}
	\item $\deg(S_n)=n$ and $\Phi_n=\Phi(S_n)$ is a path of length $\lceil (n+1)/2 \rceil$;
	
	\item $\sigma_{i,i+1}=2$ for $2\leq i\leq n-2$; 
	
	\item $\sigma_{12}=1$ and $\sigma_{n-1,n}=2$, if $n$ is odd;
	
	\item $\sigma_{12}=1=\sigma_{n-1,n}$, if $n$ is even.
\end{enumerate}

If $n=1$, take $S_1=y_1x_1x_2y_2$. If $n=2$, consider $S_2$ with $(K^{(2)}, I_2)=(\{x_1,x_2\}, \{y_1, y_2, y_3\})$, $N_1=\{x_1\}=N_3$ and $N_2=\{x_2\}$. For $n\geq 3$, we give an inductive procedure to obtain $S_{n+1}$ from $S_n$.

Let $n=2k-2$, for some $k\geq 2$, and suppose that we have already constructed $(S_n, K^{(n)}, I_n)$. By inductive hypothesis, $\Phi_n$ is a path of length $k$ satisfying (2) and (4). Hence, $I_n=\{y_1,\ldots,y_{k+1}\}$. Now, define $(S_{n+1}, K^{(n+1)},$ $I_{n+1})$ as follows: 
\[ K^{(n+1)}=K^{(n)}\dot{\cup}\{z\}, \quad I_{n+1}=I_n, \] 
\[E(S_{n+1})=E(S_n)\dot{\cup}\{xz:x\in K^{(n)}\}\dot{\cup}\{y_{k+1}z\}. \] 
Then, conditions (1) and (3) are satisfied by $S_{n+1}$.

Let $n=2k-1$, for some $k\geq 2$, and suppose that we have already constructed $(S_n, K^{(n)}, I_n)$. By inductive hypothesis, $\Phi_n$ is a path of length $k$ satisfying (2) and (3). Hence, $I_n=\{y_1,\ldots,y_{k+1}\}$. Now, define:
\[ K^{(n+1)}=K^{(n)}, \quad I_{n+1}=I_n\dot{\cup}\{y_{k+2}\}, \] 
\[E(S_{n+1})=E(S_n)\dot{\cup}\{xy_{k+2}:x\in N_i, i\leq k\}. \]
Then, conditions (1) and (4) are satisfied by $S_{n+1}$.


\section{Conclusion}\label{sec_conclusion}

In this work we investigated structural properties of the factor graph $\Phi(S)$ associated with a split graph $S$. Our analysis focused on the behavior of induced paths and cycles in $\Phi(S)$ and on the restrictions that such configurations impose on the neighborhoods of the corresponding vertices in the original split graph.

We showed that induced paths in $\Phi(S)$ generate chains of neighborhood inclusions in $S$, which in turn force a monotone behavior of the degrees of the corresponding vertices along the path. As a consequence of these structural constraints, we proved that every induced cycle in $\Phi(S)$ has length at most four. We also established that, in any induced path of $\Phi(S)$, only the first or the last edge can be simple. This property allowed us to derive an upper bound for the diameter of $\Phi(S)$ in terms of the 2-switch-degree of $S$.

Beyond these specific results, the paper provides several structural tools that help determine whether a given multigraph can arise as $\Phi(S)$ for some split graph $S$. In this sense, our results contribute toward the broader goal of recognizing factor graphs of split graphs.

There remain many open directions for future research. In particular, it would be interesting to further investigate structural properties of $\Phi(S)$ and to develop a deeper understanding of the constraints that characterize this class of multigraphs. Ultimately, one may hope for a complete characterization of factor graphs of split graphs, possibly in terms of forbidden subgraphs or other structural conditions.


\section*{Acknowledgements}
This work was partially supported by Universidad Nacional de San Luis, grants PROICO 03-0723 and PROIPRO 03-2923, MATH AmSud, grant 22-MATH-02, Consejo Nacional de Investigaciones
Cient\'ificas y T\'ecnicas grant PIP 11220220100068CO and Agencia I+D+I grants PICT 2020-00549 and PICT 2020-04064.

	
\bibliographystyle{abbrv}
\bibliography{citas__induced_paths_and_cycles_in_factor_graphs_of_split_graphs}	
	
\end{document}